\documentclass{article}
\usepackage{amssymb,amsmath,amsthm,verbatim,IEEEtrantools}
\usepackage{ amssymb, amsmath, amsthm,  graphicx, psfrag}
\usepackage[usenames]{color}
\usepackage{enumerate}

\newtheorem{mythm}{Theorem}[section]
\newtheorem{mydef}[mythm]{Definition}
\newtheorem{myex}[mythm]{Example}
\newtheorem{mycor}[mythm]{Corollary}

\newtheorem{mylem}[mythm]{Lemma}

\newtheorem{myrmk}[mythm]{Remark}

\newcommand{\R}{{\mathbb R}}

\newcommand{\p}{\partial}

\def\norm#1{\left\|#1\right\|}

\def\ba{\begin{align}}  %  need separate \end{align}
\def\ea{\end{align}}

\def\bas{\begin{align*}}  %  need separate \end{align}
\def\eas{\end{align*}}

\def\bi{\begin{itemize}} %for numbering and labeling \rabel{hey} \ref{hey} \eqref{hey}
\def\ei{\end{itemize}}  

\def\bn{\begin{enumerate}}  % need \usepackage{enumerate} 
\def\en{\end{enumerate}}    % for \bn[1.] \item \en   or   \bn[a.]  or  \bn[I/]

\def\bN{\begin{note}} 
\def\eN{\end{note}} 

\def\bl{\begin{mylabel}} 
\def\el{\end{mylabel}} 

\def\bc{\begin{cases}} 
\def\ec{\end{cases}}

\def\be{\begin{equation}} %for numbering and labeling \rabel{hey} \ref{hey} \eqref{hey}
\def\ee{\end{equation}}

\def\bes{\begin{equation*}} 
\def\ees{\end{equation*}}

\def\hat#1{{\widehat {#1}}}

\title{Kakeya-Type Sets in Local Fields with Finite Residue Field}
\author{Robert Fraser \\
Department of Mathematics,\\
University of British Columbia,\\
Vancouver, BC,\\
V6T 1Z2
}
\begin{document}
\maketitle
\begin{abstract}
We present a construction of a measure-zero Kakeya-type set in a finite-dimensional space $K^d$ over a local field with finite residue field. The construction is an adaptation of the ideas appearing in \cite{S87} and \cite{W04}. The existence of measure-zero Kakeya-type sets over discrete valuation rings is also discussed, giving an alternative construction to the one presented in \cite{DH13} over $\mathbb{F}_q[[t]]$.
\end{abstract}
\section{Introduction and Background}
A subset $T$ of Euclidean $n$-dimensional space is called a \emph{Kakeya set} if it contains a line segment of unit length in every direction. The existence of Kakeya sets with Lebesgue measure zero is shown by Besicovitch in \cite{B28}. Since the publication of Besicovitch's paper, Kakeya sets and their properties have been important in the study of harmonic analysis. For example, the existence of measure-zero Kakeya sets is central in Fefferman's celebrated result \cite{F71} that the ball multiplier is not bounded on $L^p(\mathbb{R}^n)$ for any $n \geq 2$ and any $p \neq 2$. It is conjectured that a Kakeya set in $\mathbb{R}^n$ must have Hausdorff dimension $n$. The conjecture is proven for $n = 2$ by Davies in \cite{D71}, but remains open for $n \geq 3$. Wolff shows in \cite{W95} that a Kakeya set in $\mathbb{R}^n$ must have Hausdorff dimension at least $\frac{n + 2}{2}$. This has been improved upon in the paper \cite{KT02} by Katz and Tao, and better bounds for the Minkowski dimension of Kakeya sets in $\mathbb{R}^n$ for various $n$ are obtained by Katz and Tao in \cite{KT02}, {\L}aba and Tao in \cite{LT01}, and Katz, {\L}aba and Tao in \cite{KLT00}.

In \cite{B28}, Besicovitch constructs a measure-zero Kakeya set $T$ in the plane $\mathbb{R}^2$ using an explicit geometric construction. This construction implies the existence of measure zero Kakeya sets in $\mathbb{R}^n$ because the product space $T \times [0,1]^{n-2}$ will be a measure-zero Kakeya set in $\mathbb{R}^n$. This construction relies on translating a family of triangles in order to make the measure of the resulting set small.

A somewhat different measure-zero Kakeya set construction, which makes little direct reference to the geometry of Kakeya sets, is given by Sawyer in \cite{S87}. Sawyer observes that it is sufficient to construct a function $\phi$ with the property that the range of $\phi(a) - a x$ as a function of $a$ has measure zero for every real number $x$; this shows that the cross-sections of the set $\{(x,y) : y = \phi(a) - a x \text{ for some $a \in \mathbb{R}$} \}$ have measure zero. In the main body of the paper, Sawyer constructs a universal function $\phi$ such that for any continuously differentiable function $f : \mathbb{R} \to \mathbb{R}$, the range of $\phi - f$ has measure zero. In particular, the function $\phi$ does not depend on $f$. This function $\phi$ can be used to construct Kakeya-like sets: for any measurable function $f(a,x)$ that is continuously differentiable in the indexing variable $a$, the set $\{(x,y) : y = f(a,x) - \phi(a) \text{ for some $a \in \mathbb{R}$} \}$ has measure zero.

A natural question that can be asked is whether Sawyer's result can be generalized to include transformations other than translations. Consider a family of surfaces $f_{x,y} :\mathbb{R}^d \to \mathbb{R}^{n-d}$, where $x \in \mathbb{R}^p$ and $y \in \mathbb{R}^q$. It is reasonable to ask if we can construct a measure zero set $T$ such that for every $x$, there is a $y$ such that $T$ contains the surface $f_{x,y}$. Wisewell answers this question for sufficiently smooth families of curves in \cite{W04}, given some conditions on $n, p, q,$ and $d$. 

The study of non-Euclidean Kakeya sets, by comparison, is much newer. In \cite{W99}, Wolff poses a finite-field version of the Kakeya problem. Over the finite field $\mathbb{F}_{\ell}$ containing $\ell$ elements, a Kakeya set in the $n$-dimensional vector space $\mathbb{F}_{\ell}^n$ is a set that contains a line in every direction. Wolff asks whether such a set must have $\gtrsim \ell^n$ elements, where the implied constant depends on $n$ but not on $\ell$. 

This question is answered affirmatively by Dvir in his influential paper \cite{D09}. This result is extended to a family of maximal function results in the paper \cite{EOT10} by Ellenberg, Oberlin, and Tao.

In \cite{EOT10}, the authors ask whether there are measure-zero Kakeya sets in the module $\mathbb{F}_{\ell}[[t]]^n$ where $\mathbb{F}_{\ell}[[t]]$ is the ring of formal power series. The construction of a measure-zero Kakeya set in this setting is described by Dummit and Hablicsek in \cite{DH13}. 

Dummit and Hablicsek's solution to this problem involves finding a function $\phi$ for which the set $\{(x,y) : y = ax - \phi(a) \text{ for some $a \in \mathbb{F}_{\ell}[[t]]$} \}$ has Lebesgue measure zero, and using a symmetry argument to take care of the lines that are not of this form. Dummit and Habliscek construct their function $\phi$ using linear algebra.

Specifically, Dummit and Hablicsek define their function $\phi(a)$ in the following manner: Suppose that
\[a = a_0 + a_1 t + a_2 t^2 + \ldots. \]
Dummit and Hablicsek define
\[\phi(a) = a_0^* + a_1^* t + a_2^* t^2 + \ldots \]
where $a_j^* = 0$ if $j = 2^n - 2$ for some $n$ and $a_j^* = a_{j+1}$ otherwise. For a fixed $x$ and $y$, the equation $y = ax - \phi(a)$ gives rise to a linear system in the $a_j$ variables. Dummit and Hablicsek show that the set of pairs $(x,y)$ for which this system has a solution in the $a_j$ variables has measure zero. Unfortunately, this method will not work in, say, the $p$-adic integers $\mathbb{Z}_p$ because the carry terms destroy the linearity of the problem.

We will extend the results of \cite{DH13} by constructing Kakeya-like sets using a different function $\phi$ based on the constructions in \cite{S87} and \cite{W04} that can be used in place of the function $\phi$ in \cite{DH13}. Furthermore, this construction is valid over any discrete valuation ring with finite residue field, such as $\mathbb{F}_q[[t]],$ or its corresponding local field. The adaptation of Sawyer's function to this setting is similar in spirit to the construction of subgroups of arbitrary Hausdorff dimension in the $p$-adic integers by Abercrombie in \cite{A94}.

The statement of the main theorem will concern a class of functions I call ``very strongly differentiable" functions; the definition is given in Section 2, Definition \ref{VeryStronglyDiff}. We need this notion of differentiability in order to get quantitative estimates on the error of a linear approximation that are, at least locally, independent of the choice of where the linear approximation was taken.

\begin{mythm}\label{LFWisewellFunc}
Let $R$ be a discrete valuation ring with residue field $\mathbb{F}_{\ell}$, the finite field with $\ell$ elements. Let $K$ be the field of fractions of $R$. There is a continuous function $\phi : K^p \to K^q$ with the following property:

Let $f(x,y) : K^p \times K^q \to K^{n-d},$ where $p \leq n-d \leq q$, be a measurable function that is very strongly differentiable in the $x$ and $y$ variables on every compact subset of $K^p \times K^q$, and such that the Jacobian $\frac{\partial f}{\partial y}$ has full rank a.e. in $x$ and $y.$
Then the set
\[ \{ f(x,\phi(x)) : x \in K^p\} \]
has measure zero.

\end{mythm}
As in \cite{S87} and \cite{W04}, we will use this function $\phi$ to construct measure-zero Kakeya-like sets containing certain transformations of $d$-dimensional surfaces. Specifically, for this function $\phi$, we have the following result:
\begin{mythm}\label{LFWisewellSet}
Let $f(x,y,w) : K^p \times K^q \times K^d \to K^{n-d}$ be a function such that $f(x, y, \cdot)$ is measurable for every $x$ and $y$ and satisfies the same differentiability properties in the $x$ and $y$ variables as in Theorem \ref{LFWisewellFunc}. Then the set
\[T  = \{ f(x,\phi(x),w) : x \in K^q, w \in K^d \}\]
has measure zero.

\end{mythm}
\section{Preliminaries}

We will discuss some of the basic properties of the algebraic structures that will be used in the proof of Theorem \ref{LFWisewellFunc}.

A principal ideal domain $R$ is called a \emph{discrete valuation ring} if it has a unique prime ideal $\mathfrak{p}$. This prime ideal is always maximal, so the quotient $R / \mathfrak{p}$ is a field, called the \emph{residue field} of $R$. In this paper, we will consider only local fields where the residue field is finite. Two important examples of discrete valuation rings are the ring $\mathbb{F}_{\ell}[[t]]$ of formal power series over $\mathbb{F}_{\ell},$ where the prime ideal is the ideal generated by the element $t$, and the $p$-adic integers $\mathbb{Z}_p$, where the prime ideal is generated by the element $p$. 

Given a discrete valuation ring $R$ and its prime ideal $\mathfrak{p}$, select a set $S$ consisting of one representative from each coset of $R / \mathfrak{p}$ satisfying the condition that $0$ is in $S$. Let $t$ be a generator of the ideal $\mathfrak{p}$. Then each element of $R$ can be expressed uniquely in the form
\begin{equation}\label{TaylorSeries}
x =  x_0 + x_1 t + x_2 t^2 + x_3 t^3 + \ldots
\end{equation}
where $x_j \in S$ for every $j$. This representation is not canonical as it depends on a choice of set $S$. 

We can define a function $v$ on $R$ called the \emph{valuation} as follows. Given $x \in R$, if $x \notin \mathfrak{p},$ define $v(x) = 0$. For $x \in \mathfrak{p}$ such that $x \notin \mathfrak{p^2}$, define $v(x) = 1$. Similarly, define $v(x)$ to be the largest $k$ for which $x$ is in $\mathfrak{p}^k$ (taking $v(0) = \infty$). It is clear from this definition that, given $a$ and $b$ in $R$, the valuation of $a - b$ is the number of consecutive identical terms of $a$ and $b$ in the power series expansions of $a$ and $b$ with respect to any set $S$ of representatives of $R / \mathfrak{p}$ containing $0$.

We can use the valuation to define a norm on $R$. For $x \in R$ with residue field $\mathbb{F}_{\ell}$, we define a norm on $R$ by $|x|=\ell^{-v(x)}$. This norm is an \emph{ultrametric} norm in that, for any $x$ and $y$ we have the inequality $|x + y| \leq \max(|x|,|y|)$.

With respect to this norm, the space $R$ is compact. So $R$ is a compact abelian group, and we can put a finite Haar measure $| \cdot |$ on the space $R$. We can normalize the measure so that $|R| = 1$. By translation invariance, it follows that $\mathfrak{p}$ and its cosets have measure $\frac{1}{\ell}$, where $\ell$ is the number of elements in the residue field $\mathbb{F}_\ell$. Similarly, the cosets of $\mathfrak{p}^k$ can be seen to have measure $\frac{1}{\ell^k}$; this defines the Haar measure $| \cdot |$ on the cylinder $\sigma$-algebra generated by cosets of these ideals. 

Given any integral domain $R$, its \emph{field of fractions} $K$ is the smallest field containing $R$. The field of fractions $K$ of a discrete valuation ring $R$ is known as a \emph{local field}. Because discrete valuation rings are integrally closed, we can recover $R$ as the ring of integers of $K$. Two examples of local fields with finite residue field are given by the field of formal Laurent series $\mathbb{F}_{\ell}((t))$ over a finite field, and the $p$-adic numbers $\mathbb{Q}_p$, with residue field $\mathbb{F}_p$. 

Every element in a local field is either an element of its ring of integers $R$ or is an inverse of an element of $R$. Consider a set $S$ of representatives of the cosets of $R$ containing $0$, and let $\mathfrak{p}$ be the prime ideal of $R$. Let $t$ be a generator of the ideal $\mathfrak{p}$. Every element $x$ of $K$ can be expressed uniquely in the form
\begin{equation}\label{LaurentSeries}
x = \sum_{j \geq M} x_j t^j 
\end{equation} 
where $M$ is an integer depending on $x$, $x_M \neq 0$, and the $x_j$ are elements of $S$. We extend the valuation on $R$ to all of $K$ by defining the valuation of the above element to be $M$, and extend the absolute value on $R$ to all of $K$ by taking $|x| = \ell^{-v(x)}$ for any $x \in K$.  With this norm, the ring of integers $R$ is the closed unit ball centered at the origin in $K$, and the prime ideal $\mathfrak{p}$ of $R$ is the open unit ball centered at the origin in $K$. With respect to the metric induced by the norm, $K$ is complete, but unbounded, and therefore $K$ is not compact. However, $K$ is locally compact, and we can define a Haar measure on $K$ as well, although it won't be finite.  We can obtain the Haar measure on $K$ by extending the Haar measure on $R$ to all of $K$. The closed ball of radius $\ell$ around the origin, consisting of all elements of $K$ with valuation at least $-1$, is a disjoint union of $\ell$ translated copies of $R$, so it has measure $\ell$. Similarly, the closed unit ball around the origin with radius $\ell^j$ has measure $\ell^j$ for any $j$. 

In this paper, we will be working over the finite-dimensional space $K^n$. We give this vector space the norm
\[ || (x_1, \ldots, x_n) || = \max(|x_1| , |x_2|, \ldots, |x_n|).\]
Similarly, any matrix norms in this paper will refer to the maximum of the absolute values of the entries. 
\begin{myrmk}\label{Convergence}
With respect to this norm, a series converges if and only if the norms of the terms go to zero. This property holds in any ultrametric abelian group. See \cite{R00}, Chapter 2, Section 1.2 for details.
\end{myrmk}
We will need a notion of differentiability over $K_n$. As in $\R^n$, we define a function $f$ on $K^n$ to be differentiable if there exists a linear map $f^{\prime}(x)$ so that
\[ \norm{f(x + h) - f(x) - f^{\prime}(x) h} = o_x(||h||). \]
We call $f$ continuously differentiable if the function $f^{\prime}$ is continuous in the norm topology on $K^n$. However, when dealing with functions on vector spaces over non-archimedean local fields, continuous differentiability doesn't have the same implications as it does over $\mathbb{R}^n$. Often, we need the stronger notion of ``strict" differentiability. Let $\Omega$ be a compact subset of $K^n$. We say that $f$ is strictly differentiable on $\Omega$ if
\[ \norm{f(x + h) - f(x) - f^{\prime}(x) h} = o(||h||), \]
where the bound on the differentiability does not depend on $x$. On $\mathbb{R}^n$, this condition is equivalent to continuous differentiability.

For the proof of theorem \ref{LFWisewellSet}, we will need a better quantitative bound on the error in the linear approximation than the one given by strict differentiability. 
\begin{mydef}\label{VeryStronglyDiff}
I will call a function $f : R^n \to R^m$ very strongly differentiable on a compact set $\Omega$ if
\begin{equation}\label{VeryStronglyDiffEq}
\norm{f(x + h) - f(x) - f^{\prime}(x) h} = o(||h||^{1 + \alpha}) 
\end{equation}
for some positive number $\alpha$, and any $x$, $x + h \in \Omega$.
\end{mydef}
This quantitative version of strict differentiability is going to be used to deal with the case of functions with large derivatives. Actually, for the purposes of our proof, essentially any quantitative bound on the error in the linear approximation could be used in place of the $o(||h||^{1+ \alpha})$, but the choice of $\phi$ would depend on the specific quantitative bound being used.
\begin{myex}
An example of a function that is strictly differentiable but not very strongly differentiable on the space $R = \mathbb{Z}_p$ is given by
\[f(z p^j) = p^{j + g(j)}\]
for a suitable function $g$, and any $z$ not divisible by $p$. 
\end{myex}
\begin{proof}
If we first suppose that $|h| < |x| = p^{-j}$, This function satisfies
\[ | f(x + h) - f(x)| = |p^{j + g(j) } - p^{j + g(j)}| = 0. \]

If instead we have $p^{-j} = |x| < |h| = p^{-k}$, where $g(k) > g(j)$, then 
\begin{IEEEeqnarray*}{rCl}
| f(x + h) - f(x)| & = & |p^{k + g(k)} - p^{j + g(j)}| \\
& = & |p^{k + g(k) }|\\
& = & p^{-k - g(k) } \\
& = & |h| |h|^{g(k)/k}.
\end{IEEEeqnarray*}
We can guarantee that this is $o(h)$ by picking $g(k)/k$ appropriately. In particular, we can select $g(k)/k$ to make $|h|^{g(k)/k} \approx \left| \frac{1}{\log(|h|)} \right|$ by selecting $g(k) \approx \frac{\log k + \log \log p}{\log p}$. For such $x$ and $h$, this is $o(|h|)$ but not $o(|h|^{1 + \alpha})$ for any positive number $\alpha$. Since the same estimate as above holds, but with inequality, for $|x| = |h|$, it follows that $f$ is strictly differentiable but not very strongly differentiable.
\end{proof}
It is easily seen that very strong differentiability implies a H\"older condition on $f^{\prime}$: Compare (\ref{VeryStronglyDiffEq}) to
\[ \norm{f(x) - f(x + h) - f^{\prime}(x + h) (-h)} = o(||h||^{1 + \alpha}). \]
This guarantees that both $f^{\prime}(x) h$ and $f^{\prime}(x + h) h$ are within $c ||h||^{1 + \alpha}$ of $f(x + h) - f(x)$ for some constant $c$. This, in turn, implies that $f^{\prime}(x)$ and $f^{\prime}(x+h)$ are within $c ||h||^{\alpha}$. Therefore, we can conclude
\begin{myrmk}\label{Holder}
The derivative of a very strongly differentiable function is H\"older continuous to order $\alpha$ for some value $\alpha$.
\end{myrmk} 
\section{Definition of $\phi$}

Although the function $\phi$ described in Theorem \ref{LFWisewellFunc} is to be defined on all of $K$, it is sufficient to construct such a function $\hat \phi$ with domain $R$, and extend $\hat \phi$ to all of $K$. More specifically, fix a set $S$ of representatives of the cosets of $\mathfrak{p}$ in $R$ containing zero, and a generator $t$ of $\mathfrak{p}$. Then when we write a componentwise Laurent series expansion in the spirit of (\ref{LaurentSeries}) of our input $x$, if the lowest-degree term in the expansion of $x$ has degree $M < 0$, we can define $\phi$ as follows:
\[ \phi \left( \sum_{j= M}^{\infty} t^j x_j \right) = \hat \phi \left( \sum_{j= 0}^{\infty} t^j x_j \right) \]
(Here, each $x_j$ is in the set $S^p$). We will refer to the term $t^j x_j$ in the above expansion as the \emph{degree $j$ term} of $x$. The \emph{lowest-degree term} is then $t^M x_M$. If the expansion is finite, i.e. all of the $x_j$ for $j > J$ are equal to zero, then I will refer to $t^J x_J$ has the \emph{highest-degree term} of $x$.

It is then clear that if $ \{ f(x, \hat \phi(x)) : x \in R^p \}$ has measure zero for every very strongly differentiable function $f : R^p \times K^q \to K^{n-d}$, then we're done because we can apply the property of $\hat \phi$ to the translates of $f$ as well. From now on, I will abuse notation and not distinguish between $\phi$ and $\hat \phi$.

To construct $\phi$, I will first need to construct a sequence of matrices $r_n$ with entries in the space $C(R)$ of continuous functions on $R$ so that $r_n$ is dense in the space $M_{q \times p}(C(R))$ of $q$-by-$p$ matrices with elements in $C(R)$. 

Define the set $S_k \subset K$ to be the set of elements of the field $K$ such that the degree of the lowest-degree term is at least $- \log k$ and the degree of the highest degree term is at most $k$. 

Let $\Omega_k$ be the set of functions $f : R \to S_k$ such that $f(x)$ depends only on the coefficient vectors $x_0, \ldots x_k$  of $x$. In other words, $\Omega_k$ is the space of functions from $R$ to $S_k$ that are constant on open balls of radius at most $\ell^{-k}$. Because $S_k$ is a finite set and $R$ is covered by finitely many open balls of radius $\ell^{-k}$, it follows that $\Omega_k$ is a finite set. 

The $\Omega_k$ satisfy the inclusions $\Omega_k \subset \Omega_{k + 1}$ for every $k$. Let $r_j$ be an enumeration of all the $q$-by-$p$ matrices whose entries are elements of $\Omega_1$, followed by an enumeration of matrices whose entries are elements of $\Omega_2$, and so on. Then every $q$-by-$p$ matrix that occurs in the sequence $\{r_j \}$ in fact occurs infinitely many times in the sequence $r_j$ because of the nesting of the sets $\Omega_k$.

In addition to the sequence of $q$-by-$p$ matrices $r_j$, we will also need a sequence of ``projections" $p_n$. For any $x \in R$, we can expand $x$ as a power series with respect to the set $S$ of representatives as in (\ref{TaylorSeries}).

Define
\begin{IEEEeqnarray*}{rCl}
p_0(x) & = & x_0 \\
p_1(x) & = & x_1 t + x_2 t^2 \\
p_2(x) & = & x_3 t^3 + x_4 t^4 + x_5 t^5 \\
\vdots & \vdots & \vdots \\
\end{IEEEeqnarray*}
Define $\alpha(j) = \frac{j(j + 1)}{2}.$ Then for any $j \in \mathbb{N}$ and any $x \in R$, we have that the valuation of $p_j(x)$ is at least $\alpha(j)$. By abuse of notation, for a vector $x = (x^{(1)}, x^{(2)}, \ldots, x^{(p)}) \in R^p$, we will use $p_j(x)$ to denote the componentwise application of $p_j$: 
\[p_j(x) : = (p_j(x^{(1)}), \ldots, p_j(x^{(p)}))^T.\]

We will now define the function $\phi : R^p \to R^q$ in terms of the projections $p_k$ and the matrices $r_k$. Define $\phi$ by
\begin{equation}\label{phidef}
\phi(x) = \sum_{k=0}^{\infty} r_k(x) p_k(x).
\end{equation}

This sum converges because the valuation of $p_k(x)$ is increasing quadratically and the valuation of $r_k$ is decreasing slower than logarithmically- therefore, the norms of the terms go to zero, and the sum must converge by Remark \ref{Convergence}. 

Furthermore, it can be seen that $\phi$ is continuous on $R^p$ (and the extension of $\phi$ is continuous on all of $K^p$). 

\begin{mylem}
The function $\phi$ is continuous on $R^p$.
\end{mylem}
\begin{proof}
\begin{comment}
Consider the $k$th summand in \ref{phidef}. The value of $r_k(x)$ depends only on the values of $x_j$ as $j$ goes from $0$ to $\alpha(k+1) -1$ because $r_k$ is in some $\Omega_j$ for $j \leq k$, so the value of $r_k(x)$ depends only on $x_1, \ldots, x_j$, where $j \leq k \leq \alpha(k)$. The value of $p_k(x)$ depends only on the values of $x_j$ from $\alpha(k)$ to $\alpha(k+1)-1$ by definition. Therefore, the values of the first $k$ summands of $\phi(x),$ depend only on the first $\alpha(k+1)-1$ terms of $x$. Because the valuations of the summands are increasing, this means that for any $k$, there exists a $c(k)$ such that the first $k$ terms of $\phi(x)$ depend only on the  first $c(k)$ terms of $x$. In other words, the distance between $\phi(x)$ and $\phi(x + h)$ has norm less than $\ell^{-k}$ provided that $h$ has norm less than $\ell^{-c(k)}$. This is enough to guarantee that $\phi$ is continuous.
\end{comment}
Let $A > 0$. We wish to show that if $|x - y|$ is sufficiently small, then $\phi(x) - \phi(y)$ will have valuation larger than $A$.

Consider the $k$th summand in (\ref{phidef}), where $k \geq 1$ (the summand $r_0(x) p_0(x)$ depends only on $x_0$). The value of $r_k(x)$ depends only on the values of $x_j$ as $j$ goes from $0$ to $\alpha(k+1) -1$ because $r_k$ is in some $\Omega_j$ for $j \leq k$, so the value of $r_k(x)$ depends only on $x_0, x_1, \ldots, x_j$, where $j \leq k \leq \alpha(k)$. The value of $p_k(x)$ depends only on the values of $x_j$ from $\alpha(k)$ to $\alpha(k+1)-1$ by definition. Therefore, the values of the first $k$ summands of $\phi(x),$ depend only on the first $\alpha(k+1)-1$ terms of $x$. Furthermore, the valuation of $r_n(x) p_n(x)$ is at least $\alpha(n) - \log n$ for each $n$. This goes to $\infty$ as $n \to \infty$.

Therefore, we can pick a number $N(A)$ so that for any $n > N(A)$, the valuation of $r_n(x) p_n(x)$ will be at least $A$. If the valuation of $x - y$ is at least $\alpha(N(A)),$ then $x_0, \ldots, x_{\alpha(N(A))- 1}$ agree. So $r_k(x) p_k(x) = r_k(y) p_k(y)$ for any $k < N(A)$. This implies that $\phi(x) - \phi(y)$ has valuation at least $A$, proving the lemma.
\end{proof}

The continuity of $\phi$ is sufficient to guarantee that $\phi(R^p)$ (which is equal to $\phi(K^p)$ by periodicity) is a compact set. In particular, $\phi(R^p)$ is bounded. Actually, we can directly find a quantitative bound on $|\phi(x)|$ in a much simpler way- each summand is in $R^q$ because the valuation of $r_j(x) p_j(x)$ is always nonnegative by the selection of the $r_j$ and the $p_j$. This tells us that the norm of $\phi$ is bounded above by $1$- in other words, the range of $\phi$ is contained in $R^q$.

\section{Proof of Theorem \ref{LFWisewellFunc}}

I will show that, if $\phi$ is the function from Section 3, then for any function $f(x,y)$ that is very strongly differentiable such that $\frac{\partial f}{\partial y}$ has full rank almost everywhere, the set
\[ \{ f(x, \phi(x)) : x \in R \} \]
has measure zero.

The trick to this proof, as in \cite{S87} and \cite{W04}, is to use the differentiability conditions to approximate $f(x, \phi(x))$ by its value at some finite list of ``landmarks" using a linear approximation. We will then select a very large value of $N$ for which $r_N$ is a good approximation to $ \left. \frac{\partial f}{ \partial y}^{-1} \frac{\partial f}{\partial x} \right|_{(x, \phi(x))}$ for all $x$. The range of $f(x, \phi(x))$ will be seen to be contained in small discs around these ``landmark" points, and the dimensionality conditions will guarantee that the volume of the discs decreases more quickly than the number of landmark points increases.

Fix a natural number $A > 0$. I will show that the range of $f(x, \phi(x))$ has measure less than or equal to $\ell^{-A}$. Because the argument works for every $A$, this shows that the set has measure zero.

For $m > 0$, Define $x^{(m)}$ to be
\[x^{(m)} := \sum_{j = 0}^{m - 1} p_j(x)\]
and define 
\[\phi^{(m)}(x) := \sum_{j=0}^{m - 1} r_j(x) p_j(x).\]

We will gradually decompose $f(x, \phi(x))$ into $6$ pieces $\mathrm{I}, \ldots, \mathrm{VI}$. Pieces $\mathrm{I}$ through $\mathrm{V}$ will be shown to be small, and piece $\mathrm{VI}$ will be shown to take on only a finite number of possible values.

The first step in the decomposition is to linearly approximate $f(x, y)$ in the $x$-coordinate.
\begin{comment}
\begin{IEEEeqnarray*}{rCCl}
f(x, \phi(x))& = &  & \overbrace{f(x, \phi(x)) - f \left(x^{(N)}, \phi(x) \right) - \left. \frac{\partial f}{\partial x} \right|_{(x, \phi(x))}(x - x^{(N)})}^{I} \\
& & + & f \left(x^{(N)}, \phi(x) \right) + \left. \frac{\partial f}{\partial x} \right|_{(x, \phi(x))}(x - x^{(N)})\\
\end{IEEEeqnarray*}
\end{comment}
\[ f(x, \phi(x)) =  \mathrm{I} + f \left(x^{(N)}, \phi(x) \right) + \left. \frac{\partial f}{\partial x} \right|_{(x, \phi(x))} \left(x - x^{(N)} \right),\]
where
\[ \mathrm{I} := f(x, \phi(x)) - f \left(x^{(N)}, \phi(x) \right) - \left. \frac{\partial f}{\partial x} \right|_{(x, \phi(x))} \left(x - x^{(N)} \right). \]
Next, we will approximate $f \left(x^{(N)}, \phi(x) \right)$ linearly in the $y$-coordinate. The point $f \left(x^{(N)}, \phi^{(N)}(x) \right)$ will serve as our landmark. After this step, we are left with
\begin{comment}
\begin{IEEEeqnarray*}{Cl}
& I + \overbrace{f \left(x^{(N)}, \phi(x) \right) - f \left(x^{(N)}, \phi^{(N)}(x) \right) - \left. \frac{\partial f}{\partial y}\right|_{(x^{(N)}, \phi(x) )} (\phi(x) - \phi^{(N)}(x))}^{II}\\
+ & \overbrace{f \left(x^{(N)}, \phi^{(N)}(x) \right)}^{VI} + \left. \frac{\partial f}{\partial y}\right|_{(x^{(N)}, \phi(x) )} (\phi(x) - \phi^{(N)}(x)) + \left. \frac{\partial f}{\partial x} \right|_{(x, \phi(x))}(x - x^{(N)}).
\end{IEEEeqnarray*}
\end{comment}
\[\mathrm{I} + \mathrm{II} + \mathrm{VI} + \left. \frac{\partial f}{\partial y}\right|_{(x^{(N)}, \phi(x) )} \left(\phi(x) - \phi^{(N)}(x) \right) + \left. \frac{\partial f}{\partial x} \right|_{(x, \phi(x))} \left(x - x^{(N)} \right),\]
where
\[ \mathrm{II} := f \left(x^{(N)}, \phi(x) \right) - f \left(x^{(N)}, \phi^{(N)}(x) \right) - \left. \frac{\partial f}{\partial y}\right|_{(x^{(N)}, \phi(x) )} \left( \phi(x) - \phi^{(N)}(x) \right), \]
and
\[ \mathrm{VI} := f \left(x^{(N)}, \phi^{(N)}(x) \right). \]
It's easier to control $\frac{\p f}{\p y}$ at the point $(x, \phi(x))$ than at $(x^{(N)}, \phi(x))$, so we modify our expression again:
\begin{comment}
\begin{IEEEeqnarray*}{Cl}
& I + II + VI + \overbrace{ \left. \frac{\partial f}{\partial y}\right|_{(x^{(N)}, \phi(x) )} (\phi(x) - \phi^{(N)}(x)) - \left. \frac{\partial f}{\partial y}\right|_{(x, \phi(x))} (\phi(x) - \phi^{(N)}(x))}^{III} \\
+ & \left. \frac{\partial f}{\partial y}\right|_{(x, \phi(x))} (\phi(x) - \phi^{(N)}(x))  + \left. \frac{\partial f}{\partial x} \right|_{(x, \phi(x))}(x - x^{(N)})\\ 
\end{IEEEeqnarray*}
\end{comment}
\[\mathrm{I} + \mathrm{II} + \mathrm{III} + \mathrm{VI} + \left. \frac{\partial f}{\partial y}\right|_{(x, \phi(x))} \left(\phi(x) - \phi^{(N)}(x)\right)  + \left. \frac{\partial f}{\partial x} \right|_{(x, \phi(x))} \left(x - x^{(N)} \right), \] 
where
\[\mathrm{III} := \left. \frac{\partial f}{\partial y}\right|_{(x^{(N)}, \phi(x) )} \left(\phi(x) - \phi^{(N)}(x)\right) - \left. \frac{\partial f}{\partial y}\right|_{(x, \phi(x))} \left(\phi(x) - \phi^{(N)}(x)\right).\]
Finally, we will split the $N$th summand $r_N p_N(x)$ off of $\phi(x) - \phi^{(N)}(x)$ and the quantity $p_N(x)$ off of $x - x^{(N)}$. This allows us to write the sum in our desired form,
\begin{comment}
\begin{IEEEeqnarray*}{Cl}
& I + II + III + VI + \overbrace{\left. \frac{\p f}{\p y} \right|_{(x, \phi(x))} r_N(x) p_N(x) + \left. \frac{\p f}{\p x} \right|_{(x, \phi(x))} p_N(x)}^{V} \\
+ & \overbrace{\left. \frac{\p f}{\p y} \right|_{(x, \phi(x))} (\phi(x) - \phi^{(N+1)}(x)) + \left. \frac{ \p f}{\p x} \right|_{(x, \phi(x))} (x - x^{(N+1)})}^{IV}\\
= & I + II + III + IV + V + VI
\end{IEEEeqnarray*}
\end{comment}
\[\mathrm{I} + \mathrm{II} + \mathrm{III} + \mathrm{IV} + \mathrm{V} + \mathrm{VI},\]
where 
\[\mathrm{IV} := \left. \frac{\p f}{\p y} \right|_{(x, \phi(x))} (\phi(x) - \phi^{(N+1)}(x)) + \left. \frac{ \p f}{\p x} \right|_{(x, \phi(x))} (x - x^{(N+1)}),\]
and
\[\mathrm{V} := \left. \frac{\p f}{\p y} \right|_{(x, \phi(x))} r_N(x) p_N(x) + \left. \frac{\p f}{\p x} \right|_{(x, \phi(x))} p_N(x).\]

I will now show that for an appropriate choice of $N$, each of the terms $\mathrm{I}, \mathrm{II},\mathrm{III},\mathrm{IV},$ and $\mathrm{V}$ is small.

\begin{mylem}\label{BoundI}
Given $A > 0$, there exists $N_{\mathrm{I}}$ such that for $N \geq N_{\mathrm{I}}$, term $\mathrm{I}$ has norm no more than $\ell^{- \alpha(N) - A}$.
\end{mylem}
\begin{proof}
Pick $N_{\mathrm{I}}$ so that we have, for every $x \in R^p,$ $y \in R^q,$ and $h \in R^n : \norm{h} < \ell^{-N_{\mathrm{I}}}$, that 
\[ \norm{f(x + h, y) - f(x,y) - \left. \frac{\p f}{\p x}\right|_{(x,y)} h} \leq \ell^{-A} \norm{h}.\]

We can do this by the strict differentiability assumption on $f$, and the fact that $R^p \times R^q$ is a compact set. Then if $N > N_{\mathrm{I}}$, $\norm{x - x^{(N)}}$ is no more than  $\ell^{- \alpha(N)} \leq \ell^{-\alpha(N_{\mathrm{I}})} \leq \ell^{- N_{\mathrm{I}}}$, so the above bound applies and we have that 
\[ \norm{\mathrm{I}} \leq \ell^{-A} \ell^{- \alpha(N)}, \]
as desired.
\end{proof}
\begin{mylem}\label{BoundII}
Given $A > 0$, there exists $N_{\mathrm{II}}$ such that for $N \geq N_{\mathrm{II}}$, term $\mathrm{II}$ has norm no more than $\ell^{- \alpha(N) - A}$. 
\end{mylem}
\begin{proof}

We will select $N_{\mathrm{II}}$ to take advantage of the very strong differentiability in the $y$-variable. Specifically, we select an $N_{\mathrm{II}}$ so that, $\alpha(N) - \log(N) > N$ for all $N \geq N_{\mathrm{II}}$, and so that for any $h$ with $\norm{h}  = \ell^{-s} < \ell^{-N_{\mathrm{II}}}$, and any $x \in R^n$, $y \in R^n$, we have that
\[\norm{f(x, y + h) - f(x,y) - \left. \frac{\partial f}{\partial y}\right| _{(x,y)} h} < \ell^{-A - \log(s)} \norm{h} \]
which is possible because of the very strong differentiability- $\norm{h}^{1 + \alpha}$ is $\ell^{-(1 + \alpha)s}$, which is smaller than $\ell^{-s - \log s - A}$ if $s$ is large enough.

For any $N$, we have that $\norm{\phi^{(N)}(x) - \phi(x)}$ is smaller than $\ell^{- \alpha(N) + \log(N)}.$ So if $N > N_{\mathrm{II}}$ then the error in the linear approximation will be smaller than $\ell^{-A -\alpha(N) + \log(N) - \log(\alpha(N) - \log(N))}$, which is in turn smaller than $\ell^{-A - \alpha(N)}$ because $\alpha(N) - \log(N)$ is larger than $N$. This proves the desired bound.
\end{proof}

\begin{mylem}\label{BoundIII}
Given $A > 0$, there exists an $N_{\mathrm{III}}$ such that for $N \geq N_{\mathrm{III}}$, term $\mathrm{III}$ has norm no more than $\ell^{- \alpha(N) - A}$.
\end{mylem}
\begin{proof}
Pick $N_{\mathrm{III}}$ for the H\"older condition on $\frac{\partial f}{\partial y}$ discussed in Remark \ref{Holder} to guarantee that 
\[\norm{ \left. \frac{\partial f}{\partial y} \right|_{(x, \phi(x))} - \left.\frac{\partial f}{\partial y} \right|_{(x + h, \phi(x))} } \leq \ell^{-A} \ell^{-\log(s)} \]
if $\norm{h} = \ell^{-s}$ where $s \geq N_{\mathrm{III}}$. 

Then because we have the bound $\phi(x) - \phi^{(N)}(x) \leq \ell^{-\alpha(N) + \log(N)}$, we get the bound
\[ \norm{\mathrm{III}} \leq \ell^{-A} \ell^{- \log(\alpha(N) - \log(N))} \ell^{- \alpha(N) + \log(N)} \]
and by the same reasoning as in Lemma \ref{BoundII} this is no more than $\ell^{-\alpha(N) - A}$.
\end{proof}
\begin{mylem}\label{BoundIV}
There exists an integer $N_{\mathrm{IV}}$ such that for $N \geq N_{\mathrm{IV}}$, term $\mathrm{IV}$ has norm no more than $\ell^{- \alpha(N) - A}$.
\end{mylem}
\begin{proof}
Let $B$ be the integer such that 
\[\ell^B = \max_{\substack{x \in R^p \\ y \in R^q}} \max \left( \norm{ \left. \frac{\p f}{ \p x}\right|_{(x,y)}}, \norm{ \left. \frac{\p f}{\p y} \right|_{(x,y)}} \right).\]

Notice that $\alpha(N + 1) = \alpha(N) + N$. Pick $N_{\mathrm{IV}}$ for which $N_{\mathrm{IV}} \geq \log(N_{\mathrm{IV}} + 1) + A + B$. For any $N > N_{\mathrm{IV}}$, it follows that this same inequality holds with $N$ in place of $N_{\mathrm{IV}}$. Let $N >  N_{\mathrm{IV}}.$ Then $p_{N+j}$ has valuation larger than $\alpha(N) + A + B + \log(N + j)$ for every $j > 0$, and by assumption $r_{N + j}$ has valuation larger than $- \log(N + j)$. Therefore, performing the multiplication gives
\[ \norm{\mathrm{IV}} \leq \ell^{- \alpha(N) - A}.\]
\end{proof}
\begin{mylem}\label{BoundV}
There exist arbitrarily large values of $N$ for which the term $\mathrm{V}$ has absolute value no more than $\ell^{- \alpha(N) - A}$.
\end{mylem}
\begin{proof}
As in Lemma \ref{BoundIV}, define $B$ so that
\[\ell^B = \max_{\substack{x \in R^p \\ y \in R^q}} \max \left( \norm{ \left. \frac{\p f}{ \p x}\right|_{(x,y)}}, \norm{ \left. \frac{\p f}{\p y} \right|_{(x,y)}} \right).\]
We selected $r_n$ to be a dense sequence of $q$-by-$p$ matrices of continuous functions from $R^n$ into $K^n$. Note that $\frac{\p f}{\p y}$ has a right inverse by the assumption that $\frac{\p f}{\p y}$ has full rank and the assumption that $q \geq n - d$. Therefore, we can pick $N$ for which $r_N(x) + \left. \frac{\partial f}{\partial y}^{-1} \right|_{(x, \phi(x))} \left. \frac{\p f}{\p x} \right|_{(x, \phi(x))}$ is uniformly smaller than $\ell^{-A - B} $ as a function of $x$. $\mathrm{V}$ can be written as
\[ \mathrm{V} = \left(\left. \frac{\p f}{\p y} \right|_{(x, \phi(x))} r_N(x) + \left. \frac{\p f}{\p x} \right|_{(x, \phi(x))} \right) p_N(x).\]
$p_N(x)$ has norm no larger than $\ell^{- \alpha(N)}$ by definition of $p_N,$ and the choice of $N$ guarantees that the term multiplying $p_N(x)$ has norm no larger than $\ell^{-A}$. 
\end{proof}

Therefore, we can pick an $N$ larger than $\max( N_{\mathrm{I}}, N_{\mathrm{II}}, N_{\mathrm{III}}, N_{\mathrm{IV}})$ that satisfies the condition in Lemma \ref{BoundV}. For this value of $N$, the norm of $\mathrm{I} + \mathrm{II} + \mathrm{III} + \mathrm{IV} + \mathrm{V}$ is, by the ultrametric inequality, no more than the largest of the norms of $\mathrm{I}$, $\mathrm{II}$, $\mathrm{III}$, $\mathrm{IV}$, and $\mathrm{V}$, which is no more than $\ell^{-A - \alpha(N)}$. Therefore, the range of $f(x, \phi(x))$ is contained in balls of radius $\ell^{-A - \alpha(N)}$ centered at each point $f \left(x^{(N)}, \phi^{(N)}(x) \right)$. But the values of $x^{(N)}$ and $\phi^{(N)}(x)$ depend only on the coefficients $x_0, \ldots, x_{\alpha(N) -1}$. This means that $f(x^{(N)}, \phi^{(N)}(x))$ can only take on at most $\ell^{\alpha(N)p}$ values. So the range in of $f(x, \phi(x))$ is contained in at most $\ell^{\alpha(N)p}$ balls of radius at most $\ell^{\alpha(N) - A}$ in $K^{n-d}$. Each such ball has measure $\ell^{(\alpha(N) - A)(n-d)}$, so the total measure of the union of the balls is no more than $\ell^{(-\alpha(N) - A)(n-d)} \ell^{\alpha(N) p}$. If $n - d \geq p$, then this is no more than $\ell^{(-A(n-d))}.$ Since $A$ was arbitrary, this is sufficient to show that the range of $f(x, \phi(x))$ has measure zero.
\section{Construction of Kakeya-Like Sets}

We can use the function $\phi$ described above to construct Kakeya-like sets of the type described in Theorem \ref{LFWisewellSet}. Specifically, suppose that $f(x,y,w)$ is a function such that the cross-sections $f(x,y, \cdot)$ are measurable, and that is very strongly differentiable in the $x$ and $y$ variables, and so that the Jacobian $\frac{\p f}{\p y}$ has full rank almost everywhere. We will think of $f$ as a family of surfaces, where the $x$ and $y$ index the family of surfaces, and each surface is parameterized by the $w$ variable. We want to include one surface $\{(w, z) : z =  f(x,y,w) : w \in K^d \}$ for each $x \in K^p$. 

We claim that the set
\[ T = \{ (w, z) : z = f(x, \phi(x), w) : x \in K^p, w \in K^d \} \]
has measure zero. We know from the previous section that, for almost every $w$, the ``cross-section"  $\{f(x, \phi(x), w) : w \in K^n \}$ has measure zero. By Fubini's theorem, this implies that, if $T$ is measurable, then it must have measure zero. Therefore, the only thing that remains to be seen is that $T$ is measurable.

Note that it is enough to show that the set
\[ T^{\prime} = \{ (w, z) : z = f(x, \phi(x), w) : x \in R^p, w \in K^d \} \]
is measurable, because $T$ can be expressed as the countable union
\[ \bigcup_{a} \{ (w, z) : z = f(x + a, \phi(x + a), w) : x \in R^p, w \in K^d \} \]
for countably many translations $a$. This is useful because $R^p$ is a compact set.

We will now write $T^{\prime}$ in a way that makes its Borel measurability clear. The next thing we use is that $T^{\prime}$ can be thought of as the set of $(w,z)$ for which there exists a sequence $x_n$ in $R^p$ with $\norm{f(x_n, \phi(x_n), w) - z} < \frac{1}{n}$. This is sufficient to be in $T^{\prime}$ because the compactness of $R^p$ guarantees that a subsequence of the $x_n$ will have a limit $x$, and the continuity of the functions $f$ and $\phi$ guarantees that $\norm{f(x_n, \phi(x_n), w) - z}$ will approach zero. Furthermore, we can also impose the condition that the sequence elements $x_n$ lie in a countable dense subset $\tilde R^p$ of $R^p$. Therefore, $T^{\prime}$ can be realized as the set
\[ \bigcap_{n \in \mathbb{Z}} \bigcup_{x \in \tilde R^p} \left\{(w,z) : ||f(x, \phi(x), w) - z|| < \frac{1}{n} \right\}. \]
This set is measurable because $f(x, \phi(x), \cdot)$ is measurable for every $x$. Therefore, the set $T^{\prime}$ is measurable, and $T$ is measurable and has measure zero.

\section{Results for Discrete Valuation Rings}

The function $\phi$ defined in Section 3 turned out to be an $R^n$-valued function. This immediately proves Theorem \ref{LFWisewellFunc} in the setting of discrete valuation rings with finite residue field:
\begin{mycor}\label{DVRWisewellFunc}
Let $R$ be a discrete valuation ring with finite residue field. There is a continuous function $\phi : R^p \to R^q$ with the following property:

Let $f(x,y) : R^p \times R^q \to R^{n-d},$ where $p \leq n-d \leq q$, be a measurable function that is very strongly differentiable in the $x$ and $y$ variables and such that the Jacobian $\frac{\partial f}{\partial y}$ has full rank a.e. in $x$ and  $y$. Then the set
\[ \{ f(x,\phi(x)) : x \in R^p\} \]
has measure zero.
\end{mycor}

We can similarly prove a version of Theorem \ref{LFWisewellSet} for discrete valuation rings:

\begin{mycor}\label{DVRWisewellSet}
Let $f(x,y,w) : R^p \times R^q \times R^d \to R^{n-d}$ be a function such that $f(x, y, \cdot)$ is measurable for every $x$ and $y$, with the same differentiability properties in the $x$ and $y$ variables as in Theorem \ref{LFWisewellFunc}. Then the set
\[T := \{ f(x,\phi(x),w) : x \in R^q, w \in R^d \}\]
has measure zero.
\end{mycor}
\section{Examples}

Probably the simplest example of Theorem \ref{LFWisewellSet} and Corollary \ref{DVRWisewellSet} in action is the existence of measure-zero Kakeya sets in local fields and discrete valuation rings with finite residue field:

\begin{myex}
Given a discrete valuation ring $R$ with finite residue field, or its field of fractions $K$, applying Theorem \ref{LFWisewellSet} with $f(x,y,w) = x w - y$ gives the existence of a measure-zero set containing a line of the form $\{(w,z) : z = x w - \phi(x)\}$ for every $x \in R$ or $x \in K$.
\end{myex}
In the local field setting, the only ``direction" that has been excluded is the vertical direction $w = x$, and adding an extra line won't change the measure of the set. In the discrete valuation ring setting, we can apply the same result, interchanging the roles of $z$ and $w$, to get a measure-zero set containing a line of each direction that was ``missed" above.

In particular, this result provides an alternative construction to the one given in \cite{DH13} of a measure-zero Kakeya set in the discrete valuation ring $R = \mathbb{F}_{\ell}[[t]]$. 

Another important example, the Euclidean version of which is given in \cite{W04}, is the existence of measure-zero Nikodym-type sets where the roles of the ``direction" and ``translation" have been reversed:
\begin{myex}
Given a discrete valuation ring $R$ with finite residue field, or its field of fractions $K$, applying Theorem \ref{LFWisewellSet} with $f(x,y,w) = y w - x$ gives the existence of a measure-zero set containing a line of the form $\{(w,z) : z = \phi(x) w - x\}$ for every $x \in R$ or $x \in K$.
\end{myex}
\subsection*{Acknowledgements}

I am grateful to my supervisor Malabika Pramanik for all of her help and guidance.

This research was supported by two NSERC grants.
\bibliographystyle{myplain}
\bibliography{Formal_Power_Series_Kakeya}
\end{document}